\documentclass[preprint,12pt]{elsarticle}

\usepackage{amssymb}

\usepackage{amsmath}
\usepackage{amsthm}
\usepackage{tikz-cd}
\usepackage{graphicx}
\usepackage{bm}
\usepackage{hyperref}
\newcommand{\Hilb}{\mathrm{Hilb}}
\newcommand{\img}{\mathrm{Im}\, }
\newcommand{\projectedorange}{\mathcal{C}} 
\newcommand{\orangeO}{\mathcal{O}}
\newcommand{\RR}{\mathbb{R}}
\newcommand{\R}{\mathbb{R}}
\newcommand{\Z}{\mathbb{Z}}

\newcommand{\bb}{Bernstein-B\'ezier}
\renewcommand\geq\geqslant
\renewcommand\leq\leqslant 
\newtheorem{theorem}{Theorem}[section]
\newtheorem{corollary}[theorem]{Corollary}
\newtheorem{lemma}[theorem]{Lemma}
\newtheorem{proposition}[theorem]{Proposition}
\theoremstyle{definition}
\newtheorem{definition}{Definition}[section]
\makeatletter
\@namedef{subjclassname@2020}{
	\textup{2020} Mathematics Subject Classification}
\makeatother
\hypersetup{
	colorlinks=true, % false: boxed links; true: colored links
	linkcolor=blue, % color of internal links
	citecolor=blue, % color of links to bibliography
	filecolor=blue, % color of file links
	urlcolor=blue % color of external links
}%might not work for submission replace maybe with url

\makeatletter
\def\ps@pprintTitle{%
	\let\@oddhead\@empty
	\let\@evenhead\@empty
	\def\@oddfoot{}%
	\let\@evenfoot\@oddfoot
}
\makeatother
\begin{document}
	
	\begin{frontmatter}
		
		
			%% addressline={},
			%% city={},
			%% postcode={},
			%% state={},
			%% country={}}
		
		\title{Multivariate polynomial splines on  generalized oranges}
		
			%% addressline={},
			%% city={},
			%% postcode={},
			%% state={},
			%% country={}}
			%% addressline={},
			%% city={},
			%% postcode={},
			%% state={},
			%% country={}}
		
		\author[ohio]{Maritza Sirvent}\ead{sirventleon.1@osu.edu}
		\author[towson]{Tatyana Sorokina}\ead{tsorokina@towson.edu}
		\author[swansea]{Nelly Villamizar}\ead{n.y.villamizar@swansea.ac.uk}
		\author[swansea]{Beihui Yuan}\ead{beihui.yuan@swansea.ac.uk}
		
		\affiliation[ohio]{organization={Department of Mathematics, The Ohio State University},%Department and Organization
			addressline={231 West 18th Avenue}, 
			city={Columbus},
			postcode={OH 43210}, 
			state={Ohio},
			country={United States}}
		
		\affiliation[towson]{organization={Department of Mathematics, Towson University},%Department and Organization
			addressline={8000 York Road}, 
			city={Towson},
			postcode={MD 21252}, 
			state={Maryland},
			country={United States}} 
		
		\affiliation[swansea]{organization={Department of Mathematics, Swansea University},%Department and Organization
			addressline={Fabian Way}, 
			city={Swansea},
			postcode={SA1 8EN}, 
			country={United Kingdom}} 	  
		
		\begin{abstract}
			We consider spaces of multivariate splines defined on a particular type of simplicial partitions that we call {\it (generalized) oranges}. Such partitions are composed of a finite number of maximal faces with exactly one shared {\it medial} face. We reduce the problem of finding the  dimension of splines on oranges to computing dimensions of splines on simpler, lower-dimensional partitions that we call {\it projected oranges}. We use both algebraic and \bb~tools. 
			
		\end{abstract}
		
		
		
		
	\end{frontmatter}
	
	\section{Introduction and preliminaries} \label{introduction and Preliminaries}
	
	Let $\Delta$ be a simplicial partition of an underlying domain $\Omega$ in $\RR^k$ formed as a geometric realization of a $k$-dimensional simplicial complex. For integers $0\leq r\leq d$, the space of splines $S^r_d(\Delta)$ is defined as the set of $C^r$-smooth piecewise polynomial functions of degree at most~$d$ on the partition $\Delta$.
	The space $S^r_d(\Delta)$ is a real vector space. These spline spaces are used throughout numerical analysis and approximation theory to solve diverse problems such as the interpolation and approximation of data, design of curves and surfaces, and, in the finite element method, for the solution of differential equations, among others. 
	A fundamental question is 
	to determine the dimension of such spaces. The problem has proven to be difficult due to the dependency of the dimension on the specific geometry of $\Delta$.
	
	In 1973, Strang~\cite{Strang-1973}, published a famous conjecture on the generic dimension of $S^1_d(\Delta)$ for the bivariate case. 
	A simplicial partition $\Delta$ is said to be generic provided that for all sufficiently small perturbations of the location of the vertices of $\Delta$, the resulting decomposition $\tilde{\Delta}$ satisfies $\dim S^r_d(\tilde{\Delta})=\dim S^r_d(\Delta)$, see~\cite{ASW93} where this definition was explicitly coined. Although the term generic dimension was not explicitly mentioned by Strang, he was aware of the fact that the dimension might increase for some particular configurations.
	In \cite{AS87}, Alfeld and Schumaker prove a dimension formula for $S^r_d(\Delta)$, when the polynomial degree $d\geq 4r+1$, and $\Delta$ is a general triangulation. This result was later extended to $d\geq 3r+2$ by Hong in~\cite{Hong-1991}, and to $d\geq 3r+1$ for generic triangulations by Alfeld and Schumaker in \cite{Alfeld-Schumaker-1990}. For $r<d<3r+2$, the only dimension known explicitly is the case when $r=1$ and $d=4$, see \cite{APS87}.
	
	In 1988, Billera, see~\cite{Billera-1988}, pioneered the use of algebraic homological methods to study multivariate splines, and proved the conjecture given by Strang.
	
	For the trivariate case, Alfeld, Schumaker, and Whiteley, see~\cite{ASW93}, use the \bb~analysis and algebraic homology and rigidity theory to give an explicit expression for the generic dimension of $S^1_d(\Delta)$, for $d\geq 8$.
	
	For the spatial dimension $k\geq 3$, no explicit dimension formula is known for general triangulations. There are results on upper and lower bounds and exact dimensions for some specific triangulations. For example, in \cite{DV20a}, DiPasquale and Villamizar proved a lower bound for the dimension for splines on tetrahedral stars of a vertex. 
	A natural generalization of a star of a vertex is a star of a simplex.
	A star of a simplex $\tau\in\Delta$ is the set of simplices in $\Delta$ that contain $\tau$. If a simplicial complex $\Delta$ equals the closure of the star of $\tau$, we just call $\Delta$ a star of $\tau$. 
	In this paper, we study stars of a simplex of arbitrary dimension. We call such stars \emph{generalized oranges} or \emph{oranges}.
	
	Recall that a $k$-dimensional simplicial complex is \emph{pure} is all its maximal faces are of dimension $k$.
	\begin{definition}[Generalized orange]\label{def:orange}
		For integers $0\leq i\leq k$, a \emph{$(k,i)$-orange} is a pure $k$-dimensional simplicial complex $\mathcal{O}$ composed by $n$ maximal faces and with exactly one
		face $\tau$ of dimension $k-i$ such that every maximal face of $\mathcal{O}$ contains $\tau$. 
		We say that $\tau$ is a \emph{medial simplex}.
	\end{definition}
	Note that the two extremal cases are $i=0$ and $i=k$. For $i=0$, $\Delta=\tau$ is a simplex. For $i=k$, the medial face $\tau$ is a vertex, and $\Delta$ is a star of a vertex. 
	Examples of $(3,2), (2,1)$, and $(3,1)$-oranges are shown in
	Figs.~\ref{fig:fan}, \ref{twotriangles} and \ref{twotetra}, respectively. 
	\begin{figure}
		\centering
		\includegraphics[scale=1.1]{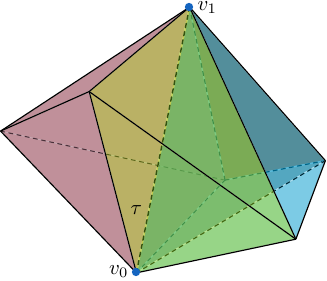}
		\caption{Example of a $(3,2)$-orange, it is a 3-dimensional simplicial complex with medial simplex $\tau$, which is the $1$-face common to all tetrahedra.}
		\label{fig:fan}
	\end{figure}
	
	Historically, the term ``orange" was introduced in~\cite{ASS92} to describe a tetrahedral partition with exactly one interior edge which is common to all tetrahedra in the partition, see Fig.~\ref{fig:fan}. 
	This matches our definition of a $(3,2)$-orange except that the medial edge does not have to be interior. The dimension for the spline spaces on a $(3,2)$-orange with interior medial edge was found in \cite{ASS92} using the trivariate cofactor method, see~\cite{LS07} for the proof as well. In~\cite{maritza}, it was noted that the dimension for the spline spaces on a $(3,2)$-orange with interior medial edge can be computed  by adding up the dimensions of bivariate splines on a 
	planar vertex star obtained by projecting the orange along the common interior edge.
	
	In this article we generalize the idea of the projection to an $i$-dimensional star of a vertex, and relate the dimension of the spline spaces over a $(k,i)$-orange $\orangeO$ to those over a projected orange $\projectedorange$, see Section \ref{section:Geometry} for the exact definition, by proving the following result.
	
	\begin{theorem}\label{theor:sumdim}
		Let $\projectedorange\subseteq \RR^i$ be the projected orange of a $(k,i)$-orange $\orangeO\subseteq \RR^k$. Then
		\begin{equation}\label{eq:main_result_formula}
			\dim S^{r}_{d}(\orangeO)=\sum_{j=0}^d\binom{d+k-i-j-1}{k-i-1}\dim S^{r}_i(\projectedorange).
		\end{equation}
	\end{theorem}
	\noindent For the two extremal cases, $i=0$ and $i=k$, Theorem \ref{theor:sumdim} trivially holds. In the case $i=0$, the orange $\orangeO=\tau$ is a simplex, $\projectedorange$ is a point, and 
	\[\dim S^{r}_{d}(\orangeO)=\sum_{j=0}^{d}\binom{d+k-j-1}{k-1}=\binom{d+k}{k}.\] For $i=k$, we know that $\orangeO=\projectedorange$. In this case, formula~\eqref{eq:main_result_formula} is just the trivial identity $\dim S^{r}_{d}(\orangeO)=\dim S^{r}_d(\projectedorange)$.
	
	The paper is organized as follows. In Section~\ref{section:Geometry}, we construct a projection $\projectedorange$ of a $(k,i)$-orange $\orangeO$ along the medial face. We prove that the projection $\projectedorange$ is also a simplicial complex. In Section~\ref{section:cofactor_on_oranges}, we analyze the relation between the space of splines on $\orangeO$ and that on $\projectedorange$ by using the so-called {\it cofactor criterion}. Finally, we prove our main result, Theorem~\ref{theor:sumdim}, by two methods: the algebraic method in Section~\ref{sec:HilbertSeries}, and the \bb~method in Section \ref{sec:BB}. The algebraic method is based on Hilbert series of spline spaces viewed as filtered vector spaces, while the \bb~method uses minimal determining sets and smoothness conditions in \bb~form.
	
	\section{Geometry}\label{section:Geometry}
	We embed a $(k, i)$-orange $\orangeO$ in $\RR^k$, and we assume that the medial face $\tau$ in $\orangeO$ of dimension $(k-i)$ is in the $(k-i)$-dimensional coordinate subspace of $\RR^{k}$ defined by $x_{1}=\cdots=x_{i}=0$. 
	As a simplifying assumption, we identify $\orangeO$ with its embedding in $\RR^k$. 
	For this choice of coordinates, we define the following projection:
	\begin{equation}
		\begin{split}\label{eq:projection}
			\pi \colon \RR^{k}&\to\RR^{i},\\
			(x_{1},\dots,x_{k})&\mapsto (x_{1},\dots,x_{i}).
		\end{split}
	\end{equation}
	Additionally, we write $\projectedorange$ for the collection of sets $\pi(\sigma)\subseteq \RR^i$ of images under the projection $\pi$ of the faces $\sigma\in\orangeO$. Namely, 
	\begin{equation}\label{eq:projOrange}
		{\projectedorange}:=\pi(\orangeO)=\bigl\{\pi(\sigma)\colon \sigma\in\orangeO\bigr\}.
	\end{equation}
	We shall refer to ${\mathcal C}$ as a \emph{projected orange}. Note that the projection of the medial face $\pi(\tau)$ is the origin $O$ of $\RR^i$. 
	For any $X,Y\subseteq \RR^k$, denote by $X*Y$ the \emph{join} of $X$ and $Y$ which is the union of all line segments joining the points in $X$ to the points in $Y$. Note that 
	\begin{equation}\label{convex}
		\pi(X*Y)=\pi(X)*\pi(Y).
	\end{equation}
	If $X=\{v\}$ and $Y=\{w\}$, we simply write $v*w$ instead of $X*Y$.
	Recall that the set of points $\{p_{1},\dots,p_{j}\}\subseteq\RR^{k}$ is said to be in \emph{general position} if its span is a $(j-1)$-subspace of $\RR^{k}$. Note that every $m$-simplex $\sigma\in\orangeO$ can be written as the join of $m+1$ points in general position.
	
	\begin{proposition}\label{prop:ProjectionIsSimplicial}
		If $\orangeO$ is a $(k,i)$-orange, then the projected orange $\projectedorange$ defined in \eqref{eq:projOrange} is a simplicial complex. Moreover, $\projectedorange$ is a star of a vertex. 
	\end{proposition}
	Before we prove Proposition~\ref{prop:ProjectionIsSimplicial}, we need some preliminary lemmas that will help us verify that every element in $\projectedorange$ is a simplex, every face of a simplex in $\projectedorange$ is also in $\projectedorange$, and that the intersection of any two elements in $\projectedorange$ is a face of each of them. 
	
	\begin{lemma}\label{lemma:lemma1_for_prop1}
		Let $\orangeO$ be a $(k,i)$-orange with the medial simplex $\tau$. The image $\pi(\sigma)$ of a simplex $\sigma\in\orangeO$ is a simplex. Moreover, if $\sigma$ and $\tau$ do not intersect, then $\pi(\sigma)$ and $\sigma$ have the same dimension. Otherwise, $\dim(\pi(\sigma))= \dim(\sigma) -\dim(\sigma\cap\tau)$.
	\end{lemma}
	\begin{proof}
		For a simplex $\sigma\in\orangeO$, the vertices of $\sigma$ can always be divided into two sets, one of which contains all vertices that are in $\tau$, the other contains those not in $\tau$. This means $\sigma$ can be written in the form 
		\begin{equation}
			\sigma=\alpha*\beta,
		\end{equation}
		where $\alpha,\beta$ are faces of $\sigma$ such that $\alpha\cap\tau=\emptyset$ and $\beta=\sigma\cap\tau$. 
		Then by~\eqref{convex}, we have 
		\begin{equation}\label{eqn:formula_pi_of_sigma}
			\pi(\sigma)=\begin{cases}
				\pi(\alpha),&\mbox{if}~\beta=\emptyset,\\
				\pi(\alpha)*O,&\mbox{if}~\beta\neq\emptyset.
			\end{cases}
		\end{equation}
		Let $z_{i+1},\dots,z_{k+1}$ be the vertices of $\tau$. 
		Without loss of generality assume that $z_{k+1}=O$. 
		Let $p_{1},\dots,p_{j}$ be the vertices of $\alpha$, where $j\leq i$. Then $p_{1},\dots,p_{j}, z_{i+1},\dots,z_{k+1}$ are in general position and $p_{1},\dots,p_{j}, z_{i+1},\dots,z_{k}$ are linearly independent in $\RR^k$. We claim that then $\bigl\{\pi(p_{1}),\dots,\pi(p_{j})\bigr\}$ is also linearly independent. 
		Note that in general the projection of a linearly independent set of vectors is not necessarily linearly independent. 
		Let $V=\mathrm{span}\{p_{1},\dots,p_{j}, z_{i+1},\dots,z_{k}\}$, and let $U=\mathrm{span}\{ z_{i+1},\dots,z_{k}\}$. 
		Then $\dim V=k+j-i$, and $\dim U=k-i$. Since $U$ is the $(k-i)$-subspace containing $\tau$, we have $\mathrm{ker} (\pi)=U$, and $\dim(\pi(V))=j$.
		Therefore, the set $\bigl\{\pi(p_{1}),\dots,\pi(p_{j})\bigr\}$ spans $\pi(V)$, and it is linearly independent.
		
		This shows that both $\pi(\alpha)$
		and $\pi(\alpha)*O$ are simplices. Moreover, if $\beta=\emptyset$, then $\dim\pi(\sigma)=\dim\pi(\alpha)$. Otherwise, $\dim(\pi(\sigma))= \dim(\sigma) -\dim(\beta)$, and the proof is complete.
	\end{proof}
	
	\begin{corollary}\label{cor:cor1_for_lemma1_for_prop1}
		The projected orange $\projectedorange$ is a collection of simplices. A face of an element in $\projectedorange$ is also in $\projectedorange$. 
	\end{corollary}
	\begin{proof}
		It immediately follows from Lemma \ref{lemma:lemma1_for_prop1} that every element of $\projectedorange$ is a simplex. 
		Now if a simplex $\omega\in \projectedorange$, then by Definition~\ref{eq:projOrange}, there is $\sigma\in\orangeO$ such that $\pi(\sigma)=\omega$. 
		We may assume $\sigma=v_{1}*\dots*v_{j}$, for vertices $v_s\in\RR^k$ and $j\leq k$. 
		Then $\pi(v_{1}*\dots*v_{j})=\pi(v_1)*\cdots*\pi(v_j)=\omega$.
		Thus any face of $\omega$ can be written as 
		$\pi(v_{j_1})*\cdots*\pi(v_{j_\ell})=\pi(v_{j_1}*\cdots*v_{j_\ell})$ and, hence, by Definition~\ref{eq:projOrange}, this face is in $\projectedorange$.
	\end{proof}
	\begin{lemma}\label{lemma:lemma2_for_prop1}
		Let $v_{1},v_{2}$ be two vertices in $\orangeO$. If neither $v_{1}$ nor $v_{2}$ belongs to $\tau$, then $\pi(v_{1})=\pi(v_{2})$ implies $v_{1}=v_{2}$.
	\end{lemma}
	\begin{proof}
		Let $\alpha_{1}=v_{1}*\tau$ and $\alpha_{2}=v_{2}*\tau$. Both $\alpha_{1}$ and $\alpha_{2}$ are $(k-i+1)$-faces in $\orangeO$ and $\tau\subseteq \alpha_{1}\cap\alpha_{2}$. Since $\orangeO$ is a simplicial complex and $\dim \tau = k-i$, then there are only two possibilities:
		\begin{equation}\label{possibilities}
			\tau=\alpha_{1}\cap\alpha_{2}\quad \text{or}\quad 
			\alpha_{1}=\alpha_{2}.
		\end{equation}
		Let $z_{i+1},\dots,z_{k+1}$ be the vertices of $\tau$, where we again assume that $z_{k+1}$ is the origin. Since $\pi(v_1)=\pi(v_2)$, we note that $v_1-v_2$ is in the $(k-i)$-dimensional coordinate space spanned by~$\tau$. Thus
		\begin{equation}\label{lincomb}
			v_1-v_2=\beta_1z_{i+1}+\dots+\beta_{k-i}z_k.
		\end{equation}
		Without loss of generality, we assume that the first $j$ coefficients $\bigl\{\beta_{\ell}\bigr\}_{\ell=1}^j$ are negative, while the remaining $k-i-j$ coefficients $\bigr\{\beta_{\ell}\bigr\}_{\ell=j+1}^{k-i}$ are non-negative. We next rewrite~\eqref{lincomb} as follows:
		\begin{equation}\label{lincomb1}
			v_1-\beta_1z_{i+1}-\dots -\beta_{j}z_{i+j}=v_2+\beta_{j+1}z_{i+j+1}+\dots+\beta_{k-i}z_k=v.
		\end{equation}
		Let $m=\left(1+\sum_{\ell=1}^{k-i}|\beta_\ell|\right)^{-1}$, and let $u=mv$. 
		
		Since $u$ can be written as two different convex combinations, one in the vertices of $\alpha_1$ and a second in the vertices of $\alpha_2$, then $u\in\alpha_1\cap\alpha_2$.
		From~\eqref{lincomb1}, it follows that $u\notin \tau$, so by~\eqref{possibilities} we must have $\alpha_1=\alpha_2$, and hence $v_1=v_2$.
	\end{proof}
	\begin{corollary}\label{cor:unique_preimage_in_link}
		Let $\omega\in \projectedorange$ such that the origin $O$ is not contained in $\omega$. Then there exists a unique $\alpha\in\orangeO$ such that $\pi(\alpha)=\omega$.
	\end{corollary}
	\begin{proof}
		The existence is immediate from the definition of $\projectedorange$. Assume there are two simplices $\alpha$ and $\alpha'$ in $\orangeO$ such that $\pi(\alpha)=\pi(\alpha')=\omega$. Since $O\notin \omega$, we have $\alpha\cap \tau=\emptyset$, and $\alpha'\cap \tau=\emptyset$. Then by Lemma \ref{lemma:lemma1_for_prop1}, $\dim(\alpha)=\dim(\alpha')=\dim(\omega)=:j$. We may assume $\alpha=p_{1}*\dots*p_{j}$ and $\alpha'=p'_{1}*\dots *p'_{j}$ for vertices $p_{s}$'s and $p'_{t}$'s in $\orangeO$. Hence $\pi(p_{1})*\dots*\pi(p_{j})=\pi(p'_{1})*\dots*\pi(p'_{j})=\omega$. Therefore, possibly after a permutation, $\pi(p_{\ell})=\pi(p'_{\ell})$ for all $\ell=1,\dots,j$. By Lemma \ref{lemma:lemma2_for_prop1}, this implies $p_{\ell}=p'_{\ell}$ for all $\ell=1,\dots,j$, and $\alpha=\alpha'$.
	\end{proof}
	\begin{corollary}\label{cor:full_preimage_not_in_link}
		Let $\omega\in \projectedorange$, and let the origin $O$ be contained in $\omega$. Then there is a unique face $\sigma\in\orangeO$ such that $\pi(\sigma)=\omega$ and $\tau\subseteq \sigma$. 
	\end{corollary}
	\begin{proof}
		By Corollary \ref{cor:cor1_for_lemma1_for_prop1}, every face $\omega$ of $\projectedorange$ is a simplex, so we may assume $\omega=\theta*O$, where $O\not\in\theta$ and $\theta$ is also a face of $\projectedorange$. 
		By Corollary \ref{cor:unique_preimage_in_link}, there is a face $\alpha\in\orangeO$ such that $\pi(\alpha)=\theta$. 
		Hence, $\pi(\alpha*\tau)=\omega$. So we can take $\sigma=\alpha*\tau$.
		To prove the uniqueness, assume there is a face $\sigma'$ such that $\pi(\sigma')=\omega$ and $\tau\subseteq \sigma'$. Then we can write $\sigma'=\alpha'*\tau$ for some $\alpha'\cap\tau=\emptyset$. Then $\pi(\alpha')=\theta$. Corollary \ref{cor:unique_preimage_in_link} implies $\alpha=\alpha'$. Hence, $\sigma=\sigma'$.
	\end{proof}
	Next corollary is an immediate consequence of Lemma \ref{lemma:lemma1_for_prop1} and Corollary \ref{cor:full_preimage_not_in_link}. We are going to use it in Section \ref{section:cofactor_on_oranges}.
	\begin{corollary}\label{cor:maximal_faces_1to1}
		There is a one-to-one correspondence between the maximal faces of an orange $\orangeO$ and its projection $\projectedorange$.
	\end{corollary}
	\begin{proof}
		If $\sigma\in \orangeO$ is a maximal face then the medial simplex $\tau\subseteq\sigma$, and so Lemma \ref{lemma:lemma1_for_prop1} implies that $\pi(\sigma)$ is a simplex of dimension $i$. 
		The existence and uniqueness of $\sigma\in\orangeO$ for every maximal face of $\projectedorange$ follow from Corollary \ref{cor:full_preimage_not_in_link}.
	\end{proof}
	\begin{lemma}\label{lemma:lemma3_for_prop1}
		For any two faces $\sigma_{1},\sigma_{2}\in\orangeO$ containing $\tau$, we have 
		\begin{equation}\label{eq:lemma5eqn}
			\pi(\sigma_{1})\cap\pi(\sigma_{2})=\pi(\sigma_{1}\cap\sigma_{2}).
		\end{equation}
		In particular, \eqref{eq:lemma5eqn} implies that $\pi(\sigma_{1})\cap\pi(\sigma_{2})\in\projectedorange$ in this case.
	\end{lemma}
	\begin{proof}
		We only need to prove $\pi(\sigma_{1})\cap\pi(\sigma_{2})\subseteq \pi(\sigma_{1}\cap\sigma_{2})$.
		Let $b\in\pi(\sigma_{1})\cap\pi(\sigma_{2})$, and let $b \neq O$. Then there exist $a_{1}\in\sigma_{1}$ and $a_{2}\in\sigma_{2}$ such that $b=\pi(a_{1})=\pi(a_{2})$. If either $a_{1}\in\sigma_{1}\cap\sigma_{2}$ or $a_{2}\in\sigma_{1}\cap\sigma_{2}$ then $b \in\pi(\sigma_{1}\cap\sigma_{2})$, and we are done. 
		Suppose that neither $a_{1}$ nor $a_{2}$ is in $\sigma_{1}\cap\sigma_{2}$. Then $(a_{1}*\tau)\cap\sigma_{2}=\tau$ and $(a_{2}*\tau)\cap\sigma_{1}=\tau$. Since $(a_{1}*\tau)\cap(a_{2}*\tau)\subseteq \sigma_{1}\cap \sigma_{2}$, we have 
		$$(a_{1}*\tau)\cap(a_{2}*\tau)=(a_{1}*\tau)\cap(a_{2}*\tau)\cap\sigma_{1}=\tau.$$
		Without loss of generality, we may assume that $x_{1}(b)=1$ and $x_{2}(b)=\dots=x_{i}(b)=0$. Note that $a_{1}*\tau$ and $a_{2}*\tau$ are in the same $(k-i+1)$-subspace $V\subseteq\RR^{k}$ defined by $x_{2}=\dots=x_{i}=0$. By a similar reasoning as in the proof of Lemma \ref{lemma:lemma2_for_prop1}, we get $\tau\neq (a_{1}*\tau)\cap(a_{2}*\tau)$, which is a contradiction. Therefore, the hypothesis that neither $a_{1}$ nor $a_{2}$ is in $\sigma_{1}\cap\sigma_{2}$ is false.
	\end{proof}
	
	\begin{corollary}\label{cor:intersection_one_in_link}
		For any two faces $\alpha,\sigma\in\orangeO$ such that $\alpha\cap\tau=\emptyset$, we have
		\begin{equation*}
			\pi(\alpha)\cap\pi(\sigma)=\pi(\alpha\cap\sigma).
		\end{equation*}
		In particular, $\pi(\alpha)\cap\pi(\sigma)\in\projectedorange$ in this case.
	\end{corollary}
	\begin{proof}
		We only need to show that $\pi(\alpha)\cap\pi(\sigma)\subseteq \pi(\alpha\cap\sigma)$.
		Note that $\sigma=\alpha'*\beta'$ such that $\alpha'\cap\tau=\emptyset$ and $\beta'\subseteq\tau$. 
		Since $\orangeO$ is a simplicial complex, then $\alpha\cap\sigma=\alpha\cap \alpha'$, and $(\alpha*\tau)\cap(\alpha'*\tau)=(\alpha\cap\alpha)*\tau$. 
		By Lemma \ref{lemma:lemma3_for_prop1}, we have that $\pi(\alpha\cap\alpha')*\pi(\tau)=(\pi(\alpha)*\pi(\tau))\cap(\pi(\alpha')*\pi(\tau))$. 
		This means that for any $b\in\pi(\alpha)\cap\pi(\sigma)$, it holds $b\in\pi(\alpha\cap\alpha')*O$. 
		So we can assume that $b$ is on the line segment $\overline{b'O}$, where $b'\in\pi(\alpha\cap\alpha')$. 
		This implies that both $b$ and $b'$ are in $\pi(\alpha)$. Since $O\not\in\pi(\alpha)$, then $b=b'$. 
		This shows $\pi(\alpha)\cap\pi(\sigma)=\pi(\alpha\cap\sigma)$.
	\end{proof}

	\begin{lemma}\label{lemma:lemma4_for_prop1}
		For any $\omega_{1},\omega_{2}\in\projectedorange$, we have $\omega_{1}\cap\omega_{2}\in\projectedorange$.
	\end{lemma}
	\begin{proof}
		If $O\not\in\omega_{1}$, then by Corollary \ref{cor:unique_preimage_in_link}, there exists a unique $\alpha\in\orangeO$ such that $\pi(\alpha)=\omega_1$. 
		Let $\sigma$ be a face in $\orangeO$ such that $\pi(\sigma)=\omega_{2}$. 
		By Corollary \ref{cor:intersection_one_in_link} we have that $\omega_{1}\cap\omega_{2}=\pi(\alpha\cap\sigma)$, and so $\omega_{1}\cap\omega_{2}\in \projectedorange$.
		
		If $O$ is in both $\omega_{1}$ and $\omega_{2}$, then by Corollary \ref{cor:full_preimage_not_in_link}, there exists $\sigma_{l}\in\orangeO$ such that $\pi(\sigma_{l})=\omega_{l}$ and $\tau\subseteq\sigma_{l}$, for $l=1,2$ respectively. 
		By Lemma \ref{lemma:lemma3_for_prop1}, it follows $\omega_{1}\cap\omega_{2}=\pi(\sigma_{1}\cap\sigma_{2})\in\projectedorange$.
	\end{proof}
	Now we are ready to prove Proposition \ref{prop:ProjectionIsSimplicial}.
	\begin{proof}[Proof of Proposition \ref{prop:ProjectionIsSimplicial}]
		The fact that $\projectedorange$ is a simplicial complex follows from Corollary \ref{cor:cor1_for_lemma1_for_prop1} and Lemma \ref{lemma:lemma4_for_prop1}. Note that $O$ is the only interior vertex of $\projectedorange$ and every facet of $\projectedorange$ contains $O$, so $\projectedorange$ must be the closure of the star of the vertex $O$.
	\end{proof}
	
	\section{Cofactors}\label{section:cofactor_on_oranges}
	In this section and thereafter, we use the notion of filtered vector spaces. 
	Recall that a \emph{filtered vector space} is a vector space $V$ with a nested sequence of subspaces $\bigl\{V_{\leq d}\subseteq V\colon d=0,1,2,\dots\bigr\}$ such that
	\begin{equation*}
		\{0\}\subseteq V_{\leq 0}\subseteq V_{\leq 1}\subseteq V_{\leq 2}\subseteq\cdots
	\end{equation*}
	and that
	\begin{equation*}
		V=\bigcup_{d\geq 0}V_{\leq d}.
	\end{equation*}
	In this article, we always assume $V_{\leq d}$ is finite dimensional for each $d$.
	Let $\Delta$ be a fixed finite $k$-dimensional triangulation of a domain $\Omega\subseteq \RR^k$. 
	If $r\geq 0$ is an integer, the \emph{total spline space} over $\Delta$ is defined as
	\begin{equation}\label{eq:def_total_spline_space}
		S^{r}(\Delta)=\bigcup_{d\geq 0}S^{r}_{d}(\Delta).
	\end{equation}
	In particular, $S^{r}(\Delta)$ is a filtered vector space. Denote by $\Delta_j$ the set of $j$-dimensional faces (or $j$-faces) of $\Delta$, for $0\leq j\leq k$. 
	We say that two $k$-faces $\sigma,\sigma'\in\Delta_k$ are \emph{maximal adjacent faces}, or simply \emph{adjacent}, if their intersection is a $(k-1)$-face of $\Delta$ i.e., if $\sigma$ and $\sigma'$ share (or have in common) a $(k-1)$-face of $\Delta$.
	
	We recall from~\cite{Billera-1988} that 
	$f=(f_\sigma\colon \sigma\in\Delta_k)\in S^{r}(\Delta)$ if and only if for every pair $\sigma,\sigma'\in\Delta_k$ such that $\sigma\cap\sigma'=\varepsilon\in\Delta_{k-1}$ there exists a polynomial $c_{\varepsilon}$ such that $f_{\sigma}-f_{\sigma'} = c_{\varepsilon}\cdot \ell_{\varepsilon}^{r+1}$, 
	where $\ell_{\varepsilon}$ is the linear polynomial defining the hyperplane containing $\varepsilon$.
	This result is called the \emph{cofactor criterion} for $C^r$-splines. 
	
	Throughout this section, we follow the notation introduced in Section \ref{section:Geometry}, and denote by $\orangeO$ a $(k,i)$-orange.
	We embedded $\orangeO$ in $\RR^{k}$, and up to a change of coordinates assume that the medial simplex $\tau$ of $\orangeO$ satisfies $x_{1}=\cdots=x_{i}=0$. 
	In these coordinates, we consider the projection $\pi$ along $\tau$ given in \eqref{eq:projection}. 
	We denote by $\mathcal{C}$ the image of $\orangeO$ by $\pi$, as defined in \eqref{eq:projOrange}.
	
	The orange $\orangeO$ is by definition a $k$-dimensional simplicial complex (see Definition \ref{def:orange}), and Proposition \ref{prop:ProjectionIsSimplicial} implies that $\mathcal{C}$ is an $i$-dimensional simplicial complex.
	Following the notation above, we denote by $\orangeO_k$ the set of $k$-dimensional (or maximal) faces of $\orangeO$.
	
	The following proposition shows how the total $C^r$-spline spaces over $\orangeO$ and $\mathcal{C}$ are related.
	\begin{proposition}\label{Prop:tensor_product_structure_on_oranges}
		For any integer $r\geq 0$, the total spline space over $\orangeO$ satisfies
		\begin{equation}\label{eq:tensor_product_structure_on_oranges}
			S^{r}(\orangeO)\simeq S^{r}(\projectedorange)\otimes_{\RR}\RR[x_{i+1},\dots,x_{k}]
		\end{equation}
		as filtered vector spaces, that is, for each degree $d$
		\begin{equation*}
			S^{r}_{d}(\orangeO)\simeq (S^{r}(\projectedorange)\otimes_{\RR}\RR[x_{i+1},\dots,x_{k}])_{\leq d},
		\end{equation*}
		where $(S^{r}(\projectedorange)\otimes_{\RR}\RR[x_{i+1},\dots,x_{k}])_{\leq d}$ is the space spanned by 
		\[\bigl\{(g,h)\in S^{r}(\projectedorange)\otimes_{\RR}\RR[x_{i+1},\dots,x_{k}]\colon \deg g+\deg h\leq d\bigr\}.\]
	\end{proposition}
	\begin{proof}
		By Corollary \ref{cor:maximal_faces_1to1}, the projection $\pi$ establishes a bijective correspondence between the maximal faces of $\orangeO$ and those of $\projectedorange$.
		Hence, every spline $g\in S^{r}(\projectedorange)$ can be written as the tuple $g=\bigl(g_{\pi(\sigma)}\colon \sigma\in \orangeO_{k}\bigr)$, where $g_{\pi(\sigma)}=g|_{\pi(\sigma)}\in\RR[x_{1},\dots,x_{i}]$. 
		Let $R=\RR[x_{1},\dots,x_{k}]$, and consider the map 
		\begin{align*}
			\varphi\colon (S^{r}(\projectedorange)\otimes_{\RR}\RR[x_{i+1},\dots,x_{k}])_{\leq d} & \to\bigoplus_{\sigma\in\orangeO_{k}} R,\\
			\left(g,h\right) & \mapsto\bigl(hg_{\pi(\sigma)}\colon \sigma\in\orangeO_{k}\bigr). 
		\end{align*}
		Notice that $S_c^r(\orangeO)\subseteq \bigoplus_{\sigma\in\orangeO_k}R$. 
		We want to show that $\varphi$ is an isomorphishm of vector spaces and $\img\varphi= S_d^r(\orangeO)$.
		
		It is clear that $\varphi$ is $\R$-linear and it is injective because $S^r(\projectedorange)\subseteq \R[x_1,\dots,x_i]$, so none of the variables $x_t$ for $t>i$ is involved in the polynomials $g_\pi(\sigma)$. 
		First we show that $\img\varphi\subseteq S^{r}_{d}(\orangeO)$. 
		
		Let $g\in S^{r}(\projectedorange)$. 
		Then, for any $h\in\RR[x_{i+1},\dots, x_{k}]$, and any pair of adjacent faces $\sigma,\sigma'\in\orangeO_k$ such that $\sigma\cap\sigma'=\varepsilon\in\orangeO_k$, the cofactor criterion implies $hg_{\pi(\sigma)}-hg_{\pi(\sigma')} = hc_{\pi(\varepsilon)}\cdot \ell_{\pi(\varepsilon)}^{r+1}$, 
		for some polynomial $c_{\pi(\varepsilon)}\in\RR[x_1,\dots, x_i]$, where $\ell_{\pi(\varepsilon)}$ is the linear polynomial vanishing on $\pi(\varepsilon)$.	
		Notice that by construction, the $(k-1)$-face $\varepsilon$ contains the medial simplex $\tau$ of $\orangeO$. 
		Consequently, the linear polynomial $\ell_{\varepsilon}$ vanishing on $\varepsilon$ is in $\RR[x_{1},\dots,x_{i}]$, and hence $\ell_{\pi(\varepsilon)}=\ell_{\varepsilon}$.
		Hence, $\varphi(g,h)\in S^{r}(\orangeO)$.
		By filtering on the degree
		we have that $\deg \varphi(g,h)\leq d$.
		This shows that $\img(\varphi)\subseteq S^{r}_{d}(\orangeO)$.
		
		We now prove that $\mathrm{im}\varphi\supseteq S^{r}_{d}(\orangeO)$. 
		Let $f=(f_\sigma\colon \sigma\in\orangeO_k)\in S^{r}_{d}(\orangeO)$, where $f|_\sigma=f_\sigma\in R$ and $\deg f_{\sigma}\leq d$.
		Notice that we can rewrite each $f_\sigma$ as a polynomial in $x_{i+1},\dots,x_k$ with coefficients in $\R[x_1,\dots,x_k]$.
		More precisely, we have
		\begin{equation*}
			f_{\sigma}=\sum_{{\bm j}\in \Z_{\geq 0}^{k-i},\;
				|\bm j|\leq d}a_{\bm j,\sigma}\,\bm{y}^{\bm j},\quad\text{for $a_{\bm j,\sigma}\in\RR[x_{1},\dots,x_{i}]$},
		\end{equation*}
		where $\bm y=(x_{i+1},\dots, x_k)$, and for a tuple of non-negative integers $\bm j=(j_{i+s})_{s=1}^{k-i}$
		we write $|\bm j|=\sum_{s=1}^{k-i}j_s$ and $\bm y^{\bm j}=x_{i+1}^{j_{i+1}}\cdots x_{k}^{j_{k}}$.
		Then, if $a_{\bm j}=\bigl(a_{\bm j,\sigma}\colon\sigma\in\orangeO_{k}\bigr)$, we can write
		\begin{equation}\label{eq:sumSplinef}
			f
			=
			\sum_{{\bm j}\in \Z_{\geq 0}^{k-i},\;
				|\bm j|\leq d}
			{a}_{\bm j}\,\bm{y}^{\bm j},
		\end{equation}
		where $\deg {a}_{\bm j}=\max\{\deg a_{\bm j,\sigma} \colon\sigma\in\orangeO_{k}\}$.
		We prove that $a_{\bm j}\in S^{r}(\projectedorange)$ for every tuple $\bm j\in\Z_{\geq 0}$ in \eqref{eq:sumSplinef} as follows. 
		Notice that for adjacent faces $\sigma$ and $\sigma'$ as above, we have that $f_{\sigma}-f_{\sigma'}=c_{\varepsilon}\cdot \ell_{\varepsilon}^{r+1}$, and we can write
		\begin{equation*}
			c_{\varepsilon}=\sum_{{\bm j}\in \Z_{\geq 0}^{k-i},\; |\bm j|\leq d}b_{\bm j,\varepsilon}\cdot\bm{\bm y}^{\bm j},\quad\text{where } b_{\bm j,\varepsilon}\in\RR[x_{1},\dots,x_{i}].
		\end{equation*}
		Thus, 
		\[
		f_{\sigma}-f_{\sigma'}
		=
		\sum_{{\bm j}\in \Z_{\geq 0}^{k-i},\; |\bm j|\leq d}(b_{\bm j,\varepsilon}\ell_{\varepsilon}^{r+1})\cdot {\bm y}^{\bm j}
		=
		\sum_{{\bm j}\in \Z_{\geq 0}^{k-i},\; |\bm j|\leq d} (a_{\bm j,\sigma}-a_{\bm j,\sigma'})\bm{y}^{\bm j}. 
		\]
		Hence, for every $\bm j\in\Z_{\geq 0}$, we have $a_{\bm j,\sigma}-a_{\bm j,\sigma'}=b_{\bm j,\varepsilon}\ell_{\varepsilon}^{r+1}$, which implies that
		$a_{\bm j}\in S^{r}(\projectedorange)$. 
		Note that $\deg f_\sigma\leq d$, so $\deg a_{\bm j}+ |\bm j|\leq d$. 
		This shows that $\img \varphi\supseteq S^{r}_{d}(\orangeO)$, and it completes the proof.
	\end{proof}
	One immediate consequence of Proposition \ref{Prop:tensor_product_structure_on_oranges} is that the algebraic structure of $S^{r}_{d}(\orangeO)$ only depends on that of $S^{r}(\projectedorange)$ and the $(k-i)$-subspace in which the medial simplex $\tau$ lies. Note that the projection map $\pi$ only depends on the $(k-i)$-subspace containing $\tau$, we have the following corollary. 
	\begin{corollary}\label{standard_orange}
		Let $\orangeO$ and $\orangeO'$ be two $(k,i)$-oranges with medial simplices $\tau$ and $\tau'$, respectively.
		If $\tau,\tau'$ lie in the same $(k-i)$-subspace and the projected oranges $\projectedorange=\projectedorange'$, then for any degree $d$,
		\begin{equation*}
			S^{r}_{d}(\orangeO)\simeq S^{r}_{d}(\orangeO').
		\end{equation*}
		In particular,
		\begin{equation*}
			\dim S^{r}_{d}(\orangeO)=\dim S^{r}_{d}(\orangeO').
		\end{equation*}
	\end{corollary}
	\begin{proof}
		By Proposition \ref{Prop:tensor_product_structure_on_oranges} it is clear that both $S^{r}_{d}(\orangeO)$ and $S^{r}_{d}(\orangeO')$ are isomorphic to $(S^{r}(\projectedorange)\otimes_{\RR}\RR[x_{i+1},\dots,x_{k}])_{\leq d}$. Therefore, they are isomorphic and have the same dimension.
	\end{proof}
	\section{Hilbert series}\label{sec:HilbertSeries}
	In this section we prove the main result in this paper which is Theorem \ref{theor:sumdim}.
	
	Note that for any order simplicial complex $\Delta$ and any order of continuity $r\geq 0$, we can see $S^{r}(\Delta)$ as a filtered vector space for the sequence of spline spaces
	$\{0\}\subseteq S_0^r(\Delta)\subseteq S_1^r(\Delta)\subseteq \cdots$. 
	The \emph{Hilbert series} $\Hilb\bigl(S^{r}(\Delta), t \bigr)$ of $S^{r}(\Delta)$ is by definition
	\begin{equation*}
		\Hilb\bigl(S^{r}(\Delta), t \bigr)=\sum_{d=0}^{\infty}\dim S^{r}_{d}(\Delta)\,t^{d}.
	\end{equation*} 
	Following the notation in previous sections, if $0\leq i\leq k$ we denote by $\orangeO$ a $(k,i)$-orange and by $\projectedorange$ the image of $\orangeO$ by the projection $\pi\colon\R^k\to \R^i$ in \eqref{eq:projection}.
	The following result is a consequence of Proposition \ref{Prop:tensor_product_structure_on_oranges} and relates the Hilbert series of the spline spaces $S^{r}(\orangeO)$ and $S^{r}(\projectedorange)$.
	\begin{corollary}\label{Prop:Prop_Hilbert_series}
		If $\orangeO$ is a $(k,i)$-orange and $\projectedorange=\pi(\orangeO)$ as defined in \eqref{eq:projOrange}, then 
		\begin{equation}\label{eqn:Hilb_series}
			\Hilb\bigl(S^{r}(\orangeO),t\bigr)
			=
			\frac{1}{(1-t)^{k-i}}\Hilb\bigl(S^{r}(\projectedorange),t\bigr).
		\end{equation}
	\end{corollary}
	\begin{proof}
		By Proposition \ref{Prop:tensor_product_structure_on_oranges}, we know that $S^{r}(\orangeO)=S^{r}(\projectedorange)\otimes_{\RR}\RR[x_{i+1},\dots,x_{k}]$ and $S^{r}_{d}(\orangeO)\cong (S^{r}(\projectedorange)\otimes_{\RR}\RR[x_{i+1},\dots,x_{k}])_{\leq d}$ for every $d\geq 0$. 
		We will prove \eqref{eqn:Hilb_series} by induction on the variables $x_{i+j}$ for $1\leq j\leq k-i$.
		Define $\phi\colon \bigl(S^{r}(\projectedorange)\otimes_\R\R[x_{i+1}]\bigr)_d\to S^{r}_d(\projectedorange)$ by taking $\phi\bigl(x_{i+1}\bigr)=0$.
		Then, $\ker_d(\phi)= \bigl(S^{r}(\projectedorange)\otimes_\R\R[x_{i+1}]\bigr)_{d-1}\cdot x_{i+1}\cong \bigl(S^{r}(\projectedorange)\otimes_\R\R[x_{i+1}]\bigr)_{d-1}$.
		This holds for every degree $d\geq 0$, and therefore it implies \begin{equation}\label{eq:corproof1}
			\Hilb(\ker(\phi),t)=\Hilb\bigl(S^{r}(\projectedorange,t)\otimes_\R\R[x_{i+1}],t\bigr)\cdot t,
		\end{equation}	
		where $\ker(\phi)$ is the filtered vector of $\ker_d(\phi)$.
		But $\phi$ is surjective and linear, hence 
		$\dim \ker_d(\phi)+\dim S^r(\projectedorange)_d=\dim \bigl(S^r(\projectedorange)\otimes_\R\R[x_{i+1}]\bigr)_{d}$, and this together with \eqref{eq:corproof1} yields
		\[
		\Hilb\bigl(S^r(\projectedorange)\otimes_\R\R[x_{i+1}],t \bigr)(1-t)=\Hilb(S^r(\projectedorange)).
		\]
		Since $\RR[x_{i+1},\dots,x_{k}]\cong \RR[x_{i+1}]\otimes_{\RR}\dots\otimes_{\RR}\RR[x_{k}]$, and $S^{r}(\orangeO)=S^{r}(\projectedorange)\otimes_{\RR}\RR[x_{i+1},\dots,x_{k}]$, we easily see by induction that $\Hilb\bigl(S^{r}(\orangeO),t\bigr)(1-t)^{k-i}=\Hilb\bigl(S^{r}(\projectedorange),t\bigr)$, which proves \eqref{eqn:Hilb_series}. 
	\end{proof}
	Theorem \ref{theor:sumdim} can be proved from Corollary \ref{Prop:Prop_Hilbert_series} as follows. 
	\begin{proof}[Proof Theorem \ref{theor:sumdim}]
		Notice that we can rewrite 
		\begin{equation*}
			\frac{1}{(1-t)^{k-i}}=\sum_{j=0}^{\infty}\binom{k-i-1+j}{j}t^{j}.
		\end{equation*}
		Thus, by Corollary \ref{Prop:Prop_Hilbert_series}, we have
		\begin{equation*}
			\sum_{d=0}^{\infty}\dim S_{d}^{r}(\orangeO)t^{d}=\left(\sum_{j=0}^{\infty}\binom{k-i-1+j}{j}t^{j}\right)\cdot\left(\sum_{j=0}^{\infty}\dim S_{j}^{r}(\projectedorange)t^{j}\right),
		\end{equation*}
		which implies 
		\begin{align*}
			\dim S_{d}^{r}(\orangeO)
			&=
			\sum_{\substack{j+l=d\\j,\,l\geq 0}}\binom{k-(i+1)+l}{l}\dim S^{r}_j(\projectedorange)\\
			&=\sum_{j=0}^d\binom{d+k-i-j-1}{k-i-1}\dim S^{r}_j(\projectedorange),
		\end{align*}
		for every degree $d\geq 0$.
	\end{proof}
	
	\section{\bb ~techniques}\label{sec:BB}
	The idea is to use Corollary~\ref{standard_orange} to transform $\orangeO$ into a special $\orangeO'$ constructed from the projection $\projectedorange$, so that its geometric structure allows to ``lift"  the \bb~basis on the projected orange $\projectedorange$ to a \bb~basis in $\orangeO'$. The main advantage of \bb~techniques is the tool called {\it domain points} that essentially replaces basis functions with actual points located in specified positions in each simplex. We assume familiarity of the reader with basic \bb~concepts, and refer to~\cite{LS07} for a comprehensive treatment of bivariate and trivariate case as well as to a survey paper~\cite{AMFO16}.
	
	Recall that by assumption the medial face $\tau$ of $\orangeO$ is in the plane $x_1=\cdots=x_i=0$. Without loss of generality assume that the first vertex $v_0$ of $\tau$ is at the origin $O$, and the remaining vertices are $v_1,\dots v_{k-i}$. We also assume that $\projectedorange$ is embedded in the subspace $x_{i+1}=\dots=x_{k}=0$, and its only interior vertex is located at $O$. We construct $\orangeO'$ as follows:
	\begin{equation*}
		\orangeO':=\projectedorange*\tau=\bigl\{\omega*v_{1}*\dots*v_{k-i}\colon \omega\in\projectedorange\bigr\}.
	\end{equation*}
	It is easy to verify that $\orangeO'$ is also a $(k,i)$-orange. It has the same medial simplex $\tau$ as $\orangeO$, and its projected orange $\projectedorange'$ equals $\projectedorange$.

	The set of all domain points in $\tau$ for a polynomial of degree at most $d$ in $(k-i)$ variables is given by:
	\begin{equation*}
		{\mathcal D}^\tau:=\bigl\{\xi^\tau_{j_0,j_1\dots,j_{k-i}}=(j_1v_1+\cdots +j_{k-i}v_{j_{k-i}})/d\colon j_0+j_1+\cdots+j_{k-i}=d\bigr\}.
	\end{equation*}
	Next we define a family of scaled versions of $\projectedorange$ as follows
	\[\projectedorange_j:=(j/d)\,\projectedorange, \;\text{for } j=0,\dots,d,\]
	where if $j=0$, the projected orange $\projectedorange$ scales to the origin. 
	Let $\omega$ be a maximal simplex in $\projectedorange$ with the first vertex $u_0$ at the origin $O$, and the remaining vertices $u_1,\dots u_i$. 
	We also define a a family of scaled versions of $\omega$ as follows
	\[\omega_j:=(j/d)\,\omega, \;\text{for } j=0,\dots,d,\]
	where if $j=0$, the simplex $\omega$ scales to the origin.
	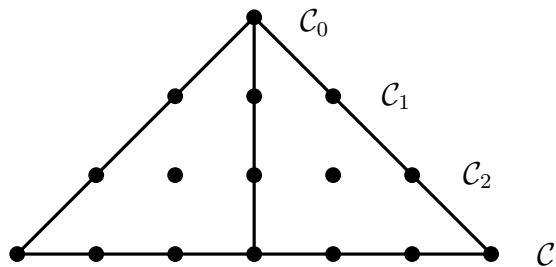
\begin{figure}[h]
		\begin{center}
			\begin{tikzpicture}[scale=0.35]
				\draw [very thick] (1,0) -- (10,0) -- (1,9) -- cycle;
				\draw [very thick] (1,9) -- (-8,0) -- (1,0);
				\fill (1,0) circle (0.3);\fill (10,0) circle (0.3);\fill (1,9) circle (0.3);\fill (4,0) circle (0.3);\fill (7,0) circle (0.3);\fill (1,6) circle (0.3);\fill (1,3) circle (0.3);\fill (4,3) circle (0.3);\fill (7,3) circle (0.3);\fill (4,6) circle (0.3);
				\fill (-2,0) circle (0.3);\fill (-5,0) circle (0.3);
				\fill (-8,0) circle (0.3);\fill (-5,3) circle (0.3);
				\fill (-2,3) circle (0.3);\fill (-2,6) circle (0.3);
				\draw (12.1,0) node {$\projectedorange$}; 
				\draw (9.5,3) node {$\projectedorange_2$}; 
				\draw (6.4,6) node {$\projectedorange_1$}; 
				\draw (3.3,8.8) node {$\projectedorange_0$}; 
			\end{tikzpicture}
			\caption{Four parallel layers of domain points forming $\mathcal{D}'$ for $S_3^1$ on a $(2,1)$-orange with the associated spaces of univariate splines $S_j^r(\projectedorange_j)$, for $j=0,1,2,3$, and $\projectedorange=[-1,0]\cup[0,1]$.}\label{twotriangles}
		\end{center}
	\end{figure}
	The set of all domain points in $\omega_j$ for a polynomial of degree at most $j$ in $i$ variables is given by:
	\begin{multline*}
		{\mathcal D}_j^\omega:=\bigl\{\xi^{\omega,j}_{\ell_0,\ell_1\dots,\ell_i}=\frac{j}{d}(\ell_1u_1+\cdots +\ell_iu_i)/j=(\ell_1u_1+\cdots +\ell_iu_i)/d\colon\\ \ell_0+\ell_1+\cdots+\ell_i=j\bigr\}.
	\end{multline*}
	The purpose of the following lemma
	is to show that the domain points in $\orangeO'$ can be split into parallel layers orthogonal to $\tau$, and, thus, the smoothness conditions for a spline in $S_d^r(\orangeO')$ can be also split into independent blocks.
	\begin{lemma}\label{detset}
		Let $\mathcal D'$ be the set of all domain points for a spline in $S_d^0(\orangeO')$. 
		Then $\mathcal D'$ can be lifted from the domain points in $\projectedorange$ as follows:
		\[\mathcal {D}'=\bigcup_{j=0}^d\bigl\{\cup_{\omega\in \projectedorange }\mathcal{D}^\omega_j+\xi_{j,j_1,j_2,\dots,j_{k-i}}\colon \xi_{j,j_1,j_2,\dots,j_{k-i}}\in \mathcal{D}^\tau\bigr\}.\]
	\end{lemma}
	\begin{proof} Let $\sigma$ be a maximal simplex in $\orangeO'$. Therefore, $\sigma=\omega*\tau$ for some maximal simplex $\omega\in\projectedorange$. Let $\eta_{m_0,\cdots,m_k}$, $m_0+\cdots m_k=d$ be an arbitrary domain point in $\mathcal{D}'\cap\sigma$. Without loss of generality, assume that the first vertex $u_0$ of $\omega$ is at the origin, the next $i$ vertices $u_1,\dots,u_i$ are in $\projectedorange$, and the last $(k-i)$ vertices $v_1,\dots, v_{k-i}$ are in $\tau$. 
		Then 
		\begin{align*}
			\eta_{m_0,\cdots,m_k}=&(m_1u_1+\cdots +m_iu_i+m_{i+1}v_1+\cdots +m_kv_{k-i})/d\\
			& =(m_1u_1+\cdots m_iu_i)/d+(m_{i+1}v_1+\cdots +m_kv_{k-i})/d\\
			&=:\xi^{\omega,j}_{m_0,m_1,\dots,m_i}+\xi^\tau_{j,m_{i+1},\dots, m_k},
		\end{align*}
		where for $j=d-m_{i+1}-\cdots-m_k$, we have $\xi^{\omega,j}_{m_0,m_1,\dots,m_i}\in \mathcal{D}^{\omega}_j$, and $\xi^\tau_{j,m_{i+1},\dots, m_k}\in\mathcal{D}^\tau$, since $j+m_{i+1}+\cdots m_{k}=d.$
		
		Conversely, for $0\leq j\leq d$, $\ell_0+\ell_1+\cdots+\ell_i=j$, and  $j+j_1+\cdots+j_{k-i}=d$, consider
		\begin{multline*}
			\xi^{\omega,j}_{\ell_0,\ell_1,\dots,\ell_i}+\xi^\tau_{j,j_1,\dots, j_{k-i}}=
			(\ell_1u_1+\cdots +\ell_iu_i)/d+(j_{1}v_1+\cdots +j_{k-1}v_{k-i})/d\\
			=\ell_1u_1+\cdots +\ell_iu_i+j_{1}v_1+\cdots +j_{k-1}v_{k-i})/d\;=\eta_{\ell_0,\ell_1,\dots,\ell_i,j_1,\dots,j_{k-i}},
		\end{multline*} 
		since $\ell_0+\ell_1+\cdots+\ell_i+j_1+\cdots+j_{k-i}=j+d-j=d$, and the simplex in $\orangeO'$ in $\RR^k$ has $(k+1)$ vertices $O,u_1,\cdots, u_i, v_1,\dots,v_{k-i}$.
	\end{proof}
	
	\begin{figure}[h]
		\begin{center}
			\begin{tikzpicture}[scale=0.6]
				\fill[fill=gray!20] (8,0.3) -- (10,1.6) -- (13.4,1.6);
				\draw [thin] (8,0.3) -- (10,1.6) -- (6.6,1.6) -- cycle;
				\fill[fill=gray!20] (8,0.3) -- (10,1.6) -- (6.6,1.6);
				
				\draw [thin] (9,2.6) -- (10,3.2) -- (11.6,3.2) -- cycle;
				\fill[fill=gray!20] (9,2.6) -- (10,3.2) -- (11.6,3.2);
				\draw [thin] (9,2.6) -- (10,3.2) -- (8.2,3.2) -- cycle;
				\fill[fill=gray!20] (9,2.6) -- (10,3.2) -- (8.2,3.2);
				\fill[fill=gray!20] (5,0) -- (15,0) -- (7,-2);
				\draw [thin] (8,0.3) -- (10,1.6) -- (13.4,1.6) -- cycle;
				\draw [very thick] (10,5) -- (7,-2);
				\draw [thick] (7,-2) -- (10,0);
				\draw [very thick] (5,0) -- (7,-2) -- (15,0); 
				\draw [thick] (10,0) -- (15,0) -- (10,5) -- cycle;
				\draw [thick] (10,5) -- (5,0) -- (10,0);
				\draw [dashed] [very thick] (5.65,-0.65) -- (12.2,-0.65);
				\fill (5.65,-0.65) circle (0.1);
				\fill (12.3,-0.65) circle (0.1);
				\fill (7.35,-0.65) circle (0.1);
				\fill (10.5,-0.65) circle (0.1);
				\draw [very thick]  (5.0,0) -- (15,0);
				\fill (5.0,0) circle (0.1);
				\fill (15.0,0) circle (0.1);
				\fill (6.7,0) circle (0.1);
				\fill (8.4,0) circle (0.1);
				\fill (11.7,0) circle (0.1);
				\fill (13.4,0) circle (0.1);
				\draw [dashed] [very thick] (6.3,-1.3) -- (9.65,-1.3);
				\fill (6.3,-1.3) circle (0.1);
				\fill (9.7,-1.3) circle (0.1);
				\fill (10,0) circle (0.15);
				\fill (9,-0.65) circle (0.15);
				\fill (8,-1.3) circle (0.15);
				\fill (7,-2) circle (0.15);
				\draw [dashed] [very thick] (7.3,1) -- (11,1); 
				\fill (7.2, 1) circle (0.1);
				\fill (11, 1) circle (0.1);
				\draw [dashed] [very thick] (6.6,1.6) -- (13.4,1.6); 
				\fill (6.6, 1.6) circle (0.1);
				\fill (8.4, 1.6) circle (0.1);
				\fill (11.7, 1.6) circle (0.1);
				\fill (13.5, 1.6) circle (0.1);
				\fill (9,1) circle (0.15);
				\fill (8,0.3) circle (0.15);
				\fill (10,1.6) circle (0.15);
				\draw [dashed] [very thick] (8.2,3.2) -- (11.7,3.2); 
				\fill (8.2,3.2) circle (0.1);
				\fill (11.8,3.2) circle (0.1);
				\fill (10,3.2) circle (0.15);
				\fill (9,2.6) circle (0.15);
				\fill (10,4.9) circle (0.15);

				\draw (21,0) node  [behind path, fill=gray!20, scale=0.7] {1 long solid segment with 7 points is $\projectedorange$}; 
				\draw (20.7,1.7) node  [behind path, fill=gray!20, scale=0.7] {2 medium dashed segments with 5 points are shifts of $\projectedorange_2$};
				\draw (19,3.5) node  [behind path, fill=gray!20, scale=0.7] {3 short dashed segments with 3 points each are shifts of $\projectedorange_1$}; 
				\draw (10.4,5.5) node {$v_1$}; 
				\draw (15.5,5.5) node  [behind path, fill=gray!20, scale=0.7] {4 points on $[v_1,v_2]$ are shifts of $\projectedorange_0$}; 
				\draw (6.5,-2.5) node {$v_2$}; 
			\end{tikzpicture}
			\caption{Domain points for $S^0_3(\orangeO')$: large dots are the domain points in $\tau$; small dots are the domain points on shifts of $\projectedorange_j$, $j=0,\dots,3$.}
			\label{twotetra}
		\end{center}
	\end{figure}
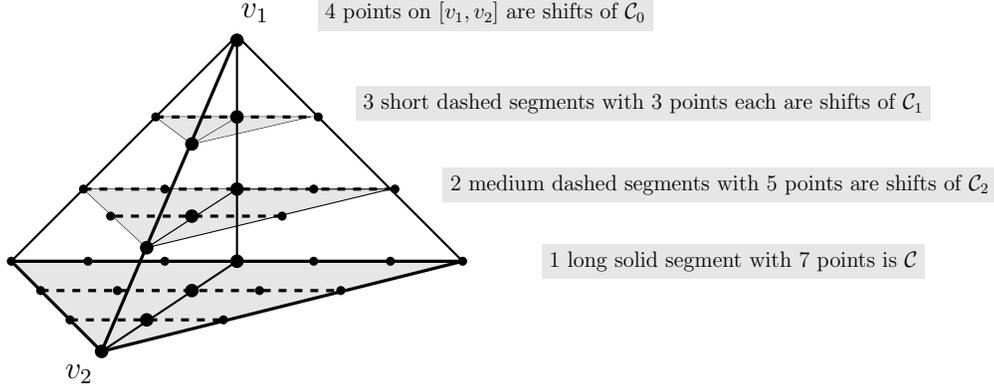
	Before proving our next result, we demonstrate the approach on two examples. The first one is a $(2,1)$-orange $\orangeO'$ depicted in~Fig.~\ref{twotriangles}. The smoothness conditions for a spline in $S_3^r(\orangeO')$ are easily seen to be the same as for a univariate spline of smoothness $r$ on the~$\projectedorange$: a partition of an interval with one node. Moreover,
	$\dim S^{r}_{3}(\orangeO')=\sum_{j=0}^3 \dim S^{r}_j(\projectedorange)$. 
	Our next example is a $(3,1)$-orange $\orangeO'$ depicted in Fig.~\ref{twotetra}. Here again, the smoothness conditions for a spline in $S_3^r(\orangeO')$ are the same as for a univariate spline of smoothness $r$ on the~$\projectedorange$: a partition of an interval with one node. Moreover, 
	\begin{equation*}
		\dim S^{r}_{3}(\orangeO')=\sum_{j=0}^3(4-j) \dim S^{r}_j(\projectedorange).
	\end{equation*}

	\begin{theorem}\label{mds} For each $0\leq j\leq d$, let ${\mathcal{M}}_j$ be a minimal determining set for a spline in $S_j^r( \projectedorange_j)$. Then the minimal determining set ${\mathcal M}'$ for a spline in $S_d^r(\orangeO')$ is given by
		\[{\mathcal M}'=\bigcup_{j=0}^d\bigl\{{\mathcal{M}}_j+\xi_{j,j_1,\dots,j_{k-i}}\colon \xi_{j,j_1,j_2,\dots,j_{k-i}}\in \mathcal{D}^\tau\bigr\}.\]
		Moreover, the dimension of $S_j^r( \projectedorange_j)$ can be computed by adding the corresponding cardinalities as follows:
		\begin{equation}\label{card}\sum_{j=0}^d \binom{d+k-j-i-1}{k-i-1} |\mathcal{M}_j|. \end{equation}
	\end{theorem}
	\begin{proof}
		Lemma~\ref{detset} implies that the domain points in $\mathcal{D}'$ lie on parallel layers that are shifts of $\projectedorange_j$ by a vector $\xi_{j,j_1,j_2,\dots,j_{k-i}}$ that is orthogonal to $\projectedorange_j$. Thus, all smoothness conditions for a spline in $S_d^r(\orangeO')$ are essentially $i$-dimensional, not $k$-dimensional, see Fig.~\ref{twotriangles} and~\ref{twotetra}. In fact, they are exactly the same as the ones for $S_j^r( \projectedorange)$ for all $j=0,\dots,d$. We now assume that we know \bb~bases for $S_j^r( \projectedorange)$ for all $j=0,\dots,d$. If we scale $\projectedorange$, these \bb~bases do not change since they are affine invariant. 
		Thus, we know the corresponding minimal determining sets ${\mathcal{M}}_j$ for a spline in $S_j^r(\projectedorange_j)$, for $j=0,\dots,d$. We assume that if $j=0$, we have one point in the corresponding MDS. 
		Thus, we only need to count the number of $\projectedorange_j$ layers to complete the proof. 
		Since $j+j_1+\cdots+j_{k-i}=d$, for each $0\leq j\leq d$, we have 
		$\binom{d+k-j-i-1}{k-i-1}$ layers of $\projectedorange_j$, and~\eqref{card} follows.
	\end{proof}
	Note that~\eqref{card} is equal to~\eqref{eq:main_result_formula}. Although it is tempting to claim that the minimal determining set $\mathcal{M}'\subseteq \orangeO'$ in Theorem~\ref{mds} gives rise to a corresponding minimal determining set in $\orangeO$, this is, of course, not true in general, since the associated change of basis matrix is not guaranteed to be non-singular. The approximation order of $S_d^r(\orangeO)$ is hard to define since this partition is not easily refinable into oranges. Assuming that we can create macro-elements based on oranges, the only conclusion that can be derived is that the approximation order of $S_d^r(\orangeO)$ cannot be better than that of $S_d^r(\projectedorange)$.
	
	\section*{Acknowledgements}
	N. Villamizar and Beihui Yuan were supported by the UK Engineering and Physical Sciences Research Council (EPSRC) New Investigator Award EP/V012835/1.
	%\bibliographystyle{plain}
	%\bibliography{references}
	
	
	
	%% \bibliographystyle{elsarticle-num} 
	%% \bibliography{<your bibdatabase>}
	

\end{document}